\documentclass[12pt]{article}
\usepackage{amssymb,amsmath,amsthm,color}

\newtheorem{formula}{}
\newtheorem{proposition}[formula]{Proposition}
\newtheorem{corollary}[formula]{Corollary}
\newtheorem{lemma}[formula]{Lemma}
\newtheorem{thm}[formula]{Theorem}
\newtheorem{example}[formula]{Example}

\newcommand{\sg}[1]{\langle{#1}\rangle}

\begin{document}
\title{The recognition problem for table algebras and reality-based algebras } 

\author{{Allen Herman\thanks{The author acknowledges the support of an NSERC Discovery Grant.}, Mikhael Muzychuk\thanks{The author acknowledges the support of the Wilson Endowment of Eastern Kentucky University}, and Bangteng Xu}}

\date{Submitted: July 1, 2015; Revised: July 27, 2016}
\maketitle

\begin{abstract}
Given a finite-dimensional noncommutative semisimple algebra $A$ over $\mathbb{C}$ with involution, we show that $A$ always has a basis $\mathbf{B}$ for which $(A,\mathbf{B})$ is a reality-based algebra.  For algebras that have a one-dimensional representation $\delta$, we show that there always exists an RBA-basis for which $\delta$ is a positive degree map. 
We characterize all RBA-bases of the $5$-dimensional noncommutative semisimple algebra for which the algebra has a positive degree map, and give examples of RBA-bases of $\mathbb{C} \oplus M_n(\mathbb{C})$ for which the RBA has a positive degree map, for all $n \ge 2$. 
\end{abstract}

\smallskip
{\small \noindent {\it Key words :} Table algebras, $C$-algebras, Reality-based algebras.  
\newline {\it AMS Classification:} Primary: 05E30; Secondary: 20C15.}

\section{Introduction} 

Let $A$ be a $(d+1)$-dimensional involutive algebra over $\mathbb{C}$, whose involution $*$ is a ring antiautomorphism that restricts to complex conjugation on scalars.  We say that the pair $(A,\mathbf{B})$ is a {\it reality-based algebra} (or RBA) if there is a basis $\mathbf{B} = \{b_0, b_1, \dots, b_d\}$ of $A$ such that 
\begin{enumerate}
\item the multiplicative identity of $A$ is an element of $\mathbf{B}$ (we index the elements of $\mathbf{B}$ so that $b_0$ is the multiplicative identity of $A$); 
\item $\mathbf{B}^2 \subseteq \mathbb{R}\mathbf{B}$, in particular the structure constants $\lambda_{ijk}$ generated by the basis $\mathbf{B}$ in the expressions $b_i b_j = \sum\limits_{k=0}^d \lambda_{ijk}b_k$ are all real numbers;
\item $\mathbf{B}^* = \mathbf{B}$, so $*$ induces a transposition on the set $\{0,1,\dots,d\}$ given by $b_{i^*}=(b_i)^*$ for all $b_i \in \mathbf{B}$; 
\item $\lambda_{ij0} \ne 0 \iff j = i^*$; and 
\item $\lambda_{ii^*0} = \lambda_{i^*i0} > 0$.
\end{enumerate}  

\noindent {\bf Remark.} In earlier treatments of reality-based algebras in the literature, the involution of the definition is assumed to be $\mathbb{C}$-linear.  Since we have a $*$-fixed basis, this is consistent here with $\bar{*}$, the composition of our involution with complex conjugation on scalars.  

\medskip
If $\mathbf{B}$ is a finite basis of an involutive algebra $A$ satisfying these properties, we will say that $\mathbf{B}$ is an {\it RBA-basis} of $A$.  If the structure constants relative to the RBA-basis $\mathbf{B}$ are integers (rational numbers), then we will say that the RBA-basis is integral (rational).  We can similarly refer to the RBA-basis as being $R$-integral for any subring $R$ of the real numbers. 

An RBA $(A,\mathbf{B})$ has a {\it degree map} if there is an algebra homomorphism $\delta: A \rightarrow \mathbb{C}$ such that $\delta(b_i)=\delta(b_i^*) \in \mathbb{R}^{\times}$ for all $b_i 
\in \mathbb{B}$. This degree map is said to be {\it positive} if $\delta(b_i)>0$ for all $b_i \in \mathbf{B}$.  When there is a positive degree map, it will be the unique algebra homomorphism $A \rightarrow \mathbb{C}$ that is positive on elements of $\mathbf{B}$.  An RBA-basis for an RBA with positive degree map is said to be {\it standard} when $\delta(b_i) = \lambda_{ii^*0}$ for all $i = 0,1,\dots,d$. 

For any algebra, a $\mathbb{C}$-linear map $\tau: A \rightarrow \mathbb{C}$ is called a {\it feasible trace} when it satisfies $\tau(xy)=\tau(yx)$ for all $x,y \in A$.  An RBA with positive degree map has a {\it standard feasible trace}, given by $\tau(\sum_i x_i b_i) = \delta(\mathbf{B}^+)x_0$ for all $\sum_i x_i b_i$ of $A$ that are expressed in terms of the basis $\mathbf{B} = \{b_0,b_1,\dots,b_d\}$.   This standard feasible trace satisfies $\tau(x^*x) > 0$ for all nonzero $x \in A$, and so it induces a nondegenerate $\mathbb{R}$-bilinear form on $A$.

For convenience we will say that the RBA-basis for an RBA with positive degree map is an {\it RBA$^{\delta}$-basis}.  A {\it table algebra} is an RBA with a positive degree map for which the structure constants with respect to its RBA-basis are all nonnegative.  We will say that the distinguished basis of table algebra is a TA-basis.  A commutative RBA with a degree map is a {\it $C$-algebra}. 

RBAs, $C$-algebras and table algebras have significant structural advantages that allow them to behave more like groups than rings.  To get an impression of this phenomenon, we direct the reader's attention to \cite{AFM}, \cite{B95}, \cite{B09}, and \cite{BX}.  It is of fundamental importance, therefore, to be able to determine whether or not a semisimple algebra over $\mathbb{C}$ has an RBA-basis, and if so, to characterize its RBA, $C$-algebra, or table algebra structures.  For commutative semisimple algebras existence of the RBA-bases is not an issue  because a finite abelian group will be a basis.  But for noncommutative algebras it requires a nontrivial construction.  For example, the algebra $M_2(\mathbb{C})$ with the conjugate-transpose involution is one example of an RBA, since 
$$ \left\{ \begin{bmatrix} 1 & 0 \\ 0 & 1 \end{bmatrix}, 
\begin{bmatrix} 0 & 1 \\ 0 & 0 \end{bmatrix},
\begin{bmatrix} 0 & 0 \\ 1 & 0 \end{bmatrix},  
\begin{bmatrix} 1 & 0 \\ 0 & -1 \end{bmatrix} \right\} $$
is an RBA-basis of $M_2(\mathbb{C})$.  Of course, $M_2(\mathbb{C})$ has no chance to have a positive degree map because it has no one-dimensional algebra representation.  (This observation was made by Blau in \cite{B09}.) 

Our main results give a full account of the existence of RBA- and RBA$^\delta$- bases for finite-dimensional semisimple algebras, and information as to whether these bases can be integral or rational.  We start by giving examples of rational RBA-bases of $M_n(\mathbb{C})$ under the conjugate-transpose involution for all $n>1$.  By applying the circle product operation we show that any semisimple algebra over $\mathbb{C}$ has a rational RBA-basis.  This is not true for semisimple algebras over $\mathbb{R}$ in general, since the real quaternion algebra with its usual involution does not have an RBA-basis. 
In the fourth section we use character theory to show that noncommutative algebras of the form $\mathbb{C} \oplus M_n(\mathbb{C})$ with $n > 1$ do NOT have integral RBA$^{\delta}$-bases.  In the fifth section we characterize all of the RBA$^{\delta}$-bases of the noncommutative $5$-dimensional algebra $\mathbb{C} \oplus M_2(\mathbb{C})$ with the conjugate-transpose involution, and give an example of a rational table algebra basis of this algebra.  In the last section we construct an RBA$^{\delta}$-bases for $\mathbb{C} \oplus M_m(\mathbb{C})$ for every $m \ge 2$ that has structure constants in the field $\mathbb{Q}(\sqrt{m})$.  

\section{Rational RBA-bases for $M_n(C)$} 

We begin by constructing examples of rational RBA-bases of the algebra $M_n(\mathbb{C})$ with respect to the conjugate-transpose involution.  Our preference is to find RBA-bases whose structure constants lie in as small a ring as possible. An integral RBA-basis is suitable for use with any coefficient ring, and a rational RBA-basis will produce an RBA structure over any field of characteristic zero. 

The first lemma will reduce the problem to the commutative subalgebra consisting of diagonal matrices.  We will write $E_{i,j}$ for the elementary matrix whose $(i,j)$-entry is $1$ and all of its other entries are $0$.

\begin{lemma} 
Suppose $\mathbf{D}$ is an RBA-basis of the commutative subalgebra of diagonal matrices in $M_n(\mathbb{C})$ for $n \ge 2$, with trivial involution.  Let $\mathbf{B}$ be the union of $\mathbf{D}$ with the set of all off-diagonal elementary matrices.  Then $\mathbf{B}$ is an RBA-basis of $M_n(\mathbb{C})$ with respect to the conjugate-transpose involution.  
\end{lemma}

\begin{proof} 
Since $E_{i,j}E_{k,\ell} = \delta_{j,k}E_{i,\ell}$ where $\delta_{j,k}$ is the Kroenecker delta, this product is either $0$ when $j \ne k$, or an off-diagonal elementary matrix in $\mathbf{B}$ when $j = k$ and $i \ne \ell$, or a non-zero diagonal elementary matrix $E_{ii}$ in the span of $\mathbf{D}$ when $j = k$ and $i = \ell$.  In the latter case any diagonal elementary matrix $E_{ii}$ is one of the primitive idempotents of the RBA $(\mathbb{C}\mathbf{D},\mathbf{D})$.  By \cite[Lemma 2.11]{B09}, the coefficient of the identity occurring in $E_{ii}$ will be a positive real number.  Thus $E_{i,j}^* = E_{j,i}$. 

Now let $D$ be one of the diagonal elements of $\mathbf{B}$.   Since the involution is trivial on $\mathbf{D}$ we have that $D = \sum_{i=1}^n p_i E_{i,i}$ with all $p_i$ real.   Let $E_{j,k}$ be an off-diagonal elementary matrix in $\mathbf{B}$.  Then we have that $E_{j,k}D = p_k E_{j,k}$ and $D E_{j,k} = p_j E_{j,k}$.  This implies that the only element $B$ of $\mathbf{B}$ for which the coefficient of $I$ in $E_{j,k}B$ or $BE_{j,k}$ will be nonzero is $E_{k,j}$.   The RBA-basis properties required for diagonal elements of $\mathbf{B}$ are inherited directly from $\mathbf{D}$.  
\end{proof}

It remains to construct rational RBA-bases of the $n$-dimensional commutative algebra $\mathcal{D}$ of diagonal $n \times n$ matrices for all $n>1$.  If $\mathbf{D} = \{b_0=1,b_1,\dots,b_{n-1}\}$ is an RBA-basis of the $n$-dimensional commutative semisimple algebra $\mathbb{C}^n$, and $\{e_0,e_1,\dots,e_{n-1}\}$ is the basis of primitive idempotents of $\mathbb{C}\mathcal{D}$, then  $b_i = \sum_j p_{i,j}e_j$ where $(p_{ij})_{i,j}$ is the first eigenmatrix.  The map $b_i \mapsto \sum_j p_{i,j} E_{j,j}$ identifies $\mathbf{D}$ with an RBA-basis of $\mathcal{D}$.  Note that the fact that the identity matrix is included in this basis is reflected by the fact that every entry of the first row of the first eigenmatrix is a $1$. 

It thus suffices to construct a table algebra of an arbitrary dimension $n$ that has a rational character table.  In dimensions up to $4$ there are association schemes that have rational character tables, which we can use to produce the following RBA-bases (here $Diag(v)$ is the diagonal matrix whose diagonal is the vector $v$):
$$\begin{array}{ll}
\mbox{Dimension $2$: }& \{ I_2, Diag(1,-1) \}, \\
\mbox{Dimension $3$: }& \{ I_3, Diag(1,-1,1), Diag(2,0,-2) \}. \mbox{ and } \\
\mbox{Dimension $4$: }& \{ I_4, Diag( 1,-1,-1,1 ), Diag( 1,-1,1,-1 ), Diag(1,1,-1,-1) \}.
\end{array}$$ 

For dimensions $5$ or more we give a construction of a table algebra that has a rational character table for the given dimension.  Let 
$n \geq 3$ be an integer.  (This constuction is also valid if $n=2$ where it constructs a Klein $4$-group.)  Define a table algebra of dimension $n+2$ with basis $\mathbf{B}=\{b_0=1, b_1,...,b_{n+1} \}$ and structure constants 
$$\begin{array}{rcl} 
b_i^2 &=& (n-1)b_0 + (n-2)b_i, \mbox{ for } i=1,...,n+1, \mbox{ and } \\
b_i b_j &=& \mathbf{B}^+ - b_i -b_j -b_0, \mbox{ for }i\neq j.
\end{array}$$
(Here we write $\mathbf{B}^+$ for $b_0 + b_1 + ... + b_{n+1}$.)  This table algebra is the Bose-Mesner algebra of the scheme corresponding to an affine plane of order $n$. So, if $n$ is a prime power this association scheme does exist, but for other values of $n$ the table algebra construction is still valid.   

Its characters are $\chi_0,\chi_1,...,\chi_{n+1}$ where $\chi_i(b_j) = -1$ if $i\neq j$ and $\chi_i(b_i) = n-1$ (here $i,j\geq 1$).  Therefore, the character table is rational.

\smallskip
The combination of the lemma with these constructions produces our first main objective. 
  
\begin{thm}\label{mnc}
For all $n > 2$, $M_n(\mathbb{C})$ with the conjugate-transpose involution has a rational RBA-basis. 
\end{thm} 

One can ask if the $4$-dimensional quaternion algebra over $\mathbb{R}$ with respect to its usual involution has an RBA-basis.  However, one of the nonidentity basis elements would have to be a non-real symmetric element with respect to the involution, and no such element exists. 

\section{Constructing an RBA-basis of a semisimple algebra}

In this section we will show that the {\it circle product} operation introduced by Arad and Fisman \cite{AF} (see also \cite{BC}) can be used to show that any semisimple involutive algebra over $\mathbb{C}$ has an RBA-basis. 

Let $(A,\mathbf{B})$ be a RBA with RBA-basis $\mathbf{B}=\{b_0,b_1,\dots,b_d\}$ and structure constants $\lambda_{ijk}$.  Suppose $\delta$ is a  linear character of $A$ that is real-valued on $\mathbf{B}$, and let $e_{\delta}$ be the corresponding centrally primitive idempotent of $A$. Let $(A_1,\mathbf{B}_1)$ be another RBA with RBA-basis $\mathbf{B}_1 = \{c_0, c_1,\dots,c_h\}$ and structure constants $\beta_{ijk}$.  The circle product $(A \circ_{\delta} A_1, \mathbf{B} \circ_
{\delta} \mathbf{B}_1)$ is defined by the following: 

\begin{enumerate}
\item $A \circ_{\delta} A_1$ is an algebra whose basis $\mathbf{B} \circ_{\delta} \mathbf{B}_1$ is the disjoint union of $\mathbf{B}$ and $\mathbf{B}_1\setminus\{c_0\}$. 

\item Considered as a product in $A \circ_{\delta} A_1$, $b_ib_j = \sum_k \lambda_{ijk} b_k$ for all $i,j \in \{0,1,\dots,d\}$. 

\item Considered as a product in $A \circ_{\delta} A_1$, $c_ic_j = \sum_k \beta_{ijk} c_k$ for $i \in \{1,\dots,h\}$ and $j \in \{1,\dots,h\}\setminus\{i^*\}$.

\item $b_ic_j = c_jb_i = \delta(b_i)c_j$ for $i \in \{0,1,\dots,d\}$ and $j \in \{1,\dots,h\}$.

\item $c_i c_{i^*} = \beta_{ii^*0}e_{\delta} + \sum_{k>0} \beta_{ijk} c_k$ for $i \in \{1,\dots,h\}$.  
\end{enumerate}

It is a consequence of \cite[Theorem 1.1]{BC} that the circle product of a $C$-algebra $(A,\mathbf{B},\delta)$ having a rational-valued degree map $\delta$ with an RBA $(A_1,\mathbf{B}_1)$ becomes an RBA whose RBA-basis is $\mathbf{B} \circ_{\delta} \mathbf{B}_1$.  From the above definition, we can see that whenever $F$ is a subfield of the real numbers for which the RBA-bases of $\mathbf{B}$ and $\mathbf{B}_1$ are both $F$-integral, then $\mathbf{B} \circ_{\delta} \mathbf{B}_1$ will be $F$-integral.   That the circle product of RBA-bases with nonnegative structure constants will be an RBA-basis with nonnegative structure constants also follows immediately from the definition.  Furthermore, if $\mathbf{B}$ and $\mathbf{B}_1$ both admit positive degree maps $\delta$ and $\delta_1$, then $\mathbf{B} \circ_{\delta} \mathbf{B}_1$ admits the positive degree map 
$$ \tilde{\delta}(b) = \begin{cases} \delta(b) & \mbox{ if } b \in \mathbf{B} \\ \delta_1(b) & \mbox{ if } b \in \mathbf{B}_1 \end{cases}, \mbox{ for all } b \in \mathbf{B} \circ_{\delta} \mathbf{B}_1. $$

We will apply the circle product operation to construct an RBA-basis of $\mathbb{C} \oplus M_n(\mathbb{C})$.  For our $C$-algebra $(A,\mathbf{B})$ we use the $2$-dimensional group algebra $\mathbb{C}[C_2]$, with degree map given by the trivial character of $C_2$.  For the RBA $(A_1,\mathbf{B}_1)$ we use $M_n(\mathbb{C})$ with a rational RBA-basis guaranteed by Theorem \ref{mnc}.

\begin{thm} 
Let $\mathbb{C}C_2$ be the complex group algebra of the group $C_2=\{1,x\}$, and let $\delta$ be the trivial character of the group $C_2$.  Let $\mathbf{B}=\{b_0,b_1,\dots,b_d\}$ be a rational RBA-basis of $M_n(\mathbb{C})$.  

Then $C_2 \circ_{\delta} \mathbf{B}$ is a rational RBA-basis of $\mathbb{C} \oplus M_n(\mathbb{C})$. 
\end{thm} 

\begin{proof}
The definition of the structure constants for the circle product basis  
$C_2 \circ_{\delta} \mathbf{B} = C \cup (\mathbf{B} \setminus \{I\})$ 
requires the centrally primitive idempotent $e_{\delta}=\frac12(1+x)$ of $\mathbb{C}C_2$.  The fact that this circle product basis is a rational RBA-basis is a consequence of \cite[Theorem 1.1]{BC}.  From the definition of the circle product in \cite{BC}, 
$$\mathbb{C}[C_2 \circ_{\delta} \mathbf{B}] = \mathbb{C}(1 - e_{\delta}) \oplus \mathbb{C}[(\mathbf{B} \setminus \{b_0\}) \cup \{e_{\delta}\}], $$
which is isomorphic as an algebra to $\mathbb{C} \oplus M_n(\mathbb{C})$ since $e_{\delta}b=be_{\delta}=1b=b1=b$, for all $b \in \mathbf{B}\setminus \{b_0\}$. 
\end{proof} 

\begin{corollary}\label{allrba}
Every finite-dimensional semisimple algebra over $\mathbb{C}$ has a rational RBA-basis. 
\end{corollary}

\begin{proof} 
Induct on the number of simple components of the semisimple finite-dimensional algebra $A$.  If $A$ is simple, then $A \simeq M_n(\mathbb{C})$ and so it has a rational RBA-basis as observed previously. 

If $A$ is not simple, let $A = M_n(\mathbb{C}) \oplus A_1$, where $A_1$ is a semisimple algebra with fewer components. Then $m = dim(A_1) < dim(A)$. By our inductive hypothesis, $A_1$ has a rational RBA-basis, call this $\mathbf{B}_1$.  By the previous theorem $M_n(\mathbb{C}) \oplus \mathbb{C}$ is isomorphic to the circle product $\mathbb{C}[C_2] \circ_{\delta} M_n(\mathbb{C})$. 
Let $\delta'$ be the real linear character of $\mathbb{C}C_2 \circ_{\delta} M_n(\mathbb{C}) \simeq \mathbb{C} \oplus M_n(\mathbb{C})$ with $\delta'(M_n(\mathbb{C}))=0$.  Then $\delta'(e_{\delta})=0$. If $\mathbf{B}$ is a rational RBA-basis for $M_n(\mathbb{C})$, then $$\mathbb{C}[(C_2 \circ_{\delta} \mathbf{B}) \circ_{\delta'} \mathbf{B}_1] \simeq M_n(\mathbb{C}) \oplus A_1,$$ 
since $e_{\delta'}b_1 - b_1e_{\delta'} = b_1$ and $bb_1=b_1b=\delta'(b)b_1=0$, for all $b \in \mathbf{B}$ and $b_1 \in \mathbf{B}_1\setminus \{1\}$. 
\end{proof}

\section{Noncommutative algebras with $|Irr(A)|=2$} 

We will require the following well-known facts concerning the character theory of RBAs.  These have appeared in various forms in the literature over the years (see for example \cite{B95}, \cite{B10}, or \cite{AFM}), but first appeared in this generality in work of Higman \cite{Hig87} describing the character theory of semisimple involutive algebras with a $*$-closed basis. 

\begin{proposition}\label{ctorba}
Let $(A,\mathbf{B})$ be an RBA with respect to the involution $*$, and suppose $\delta$ is a positive degree map on $A$.  Let $Irr(A)$ be the set of irreducible characters of $A$, and for each $\chi \in Irr(A)$, let $m_{\chi}$ be the multiplicity of $\chi$ in the standard feasible trace $\tau$ of $A$, and let $e_{\chi}$ be the centrally primitive idempotent of $A$ for which $\chi(e_{\chi})=\chi(1)>0$.  Then the following hold: 

\begin{enumerate} 
\item (Positive multiplicities) For all $\chi \in Irr(A)$, $m_{\chi} > 0$. 

\item (Idempotent character formula) For all $\chi \in Irr(A)$, $e_{\chi} = \displaystyle{ \frac{m_{\chi}}{\delta(\mathbf{B}^+)} \sum_i \frac{\chi(b_i^*)}{\lambda_{ii^*0}}b_i}$. 

\item (Orthogonality relations) For all $\chi, \psi \in Irr(A)$, $\chi(e_{\psi}) = \delta_{\chi \psi} \chi(1).$

\end{enumerate}
\end{proposition}

The positive degree map $\delta$ is an irreducible character of $A$, and $m_{\delta}=1$.   For later use, we note that since our involution extends complex conjugation on scalars, we have $e_{\chi}^* = e_{\chi}$, for all $\chi \in Irr(A)$.   When $\psi, \chi \in Irr(A)$ with $\psi \ne \chi$, the fact that $\psi(xe_{\chi})=0$, for all $x \in A$ implies that $\tau(x^*xe_{\chi}) = m_{\chi} \chi(x^*x)$, for all $x \in A$.   It then follows from Proposition \ref{ctorba}(i) that $\chi(x^*x)\ge 0$ for all $x \in A$.

The referee has remarked that the next theorem is a corollary to \cite[Theorem 1]{B09}.  The proof provided here is independent of this result.
 
\begin{thm} Let $(A,\mathbf{B})$ be a standard integral RBA with a positive degree map.   

If $|Irr(A)| = 2$, then $|\mathbf{B}|=2$.  
\end{thm}

\begin{proof} 
Let $Irr(A) = \{\delta,\chi\}$.  Let $e_{\delta}$ and $e_{\chi}$ be the two centrally primitive idempotents of $A$.  We can assume that the distinguished basis $\mathbf{B}$ is a standardized basis, so we have $\delta(b_i) = \lambda_{ii^*0}$ for $i=0,1,\dots,d$.  Let $n = \delta(\mathbf{B}^+)$ be the order of $\mathbf{B}$. Since $\mathbf{B}$ is an integral RBA basis we have that $\delta_i \in \mathbb{Z}^+$.  By Proposition \ref{ctorba}, we have that
$$ e_{\delta} = \frac{1}{n}\sum_i b_i, \mbox{ and } 
e_{\chi} = \frac{m_{\chi}}{n} \sum_i \frac{\chi(b_i^*)}{\delta_i} b_i, $$
for some positive real number $m_{\chi}$.   

Since $|Irr(A)|=2$, $e_1 + e_{\chi} = b_0$.  From this one can show that $n = 1 + m_{\chi}\chi(b_0)$, and for $b_i \ne b_0$, $\chi(b_i) = -\frac{\delta_{i^*}}{m_{\chi}}$. 
In particular, $\chi(b_i)$ is a negative rational number when $b_i \ne b_0$.  Since $\chi(b_i)$ is an algebraic integer whenever the structure constants for the basis $\mathbf{B}$ are integers, all of the $\chi(b_i)$'s are in fact integers, and so $\chi(b_i) \le -1$ for $i>0$.  

By the orthogonality relations, 
$$ 0 = \sum_i \chi(b_i) = \chi(b_0) + \sum_{i =1}^d \chi(b_i) \le  \chi(b_0) - (|\mathbf{B}| - 1). $$
Since $|\mathbf{B}| = 1 + \chi(b_0)^2$, it follows that $\chi(b_0) \ge \chi(b_0)^2$.  Since $\chi(b_0)$ is a positive integer, this forces $\chi(b_0) = 1$, and hence $|\mathbf{B}|=2$, as required. 
\end{proof}

One interpretation of the preceding result is that any noncommutative semisimple algebra that has an integral RBA$^{\delta}$-basis must have at least 3 simple components.  In particular, the noncommutative $5$-dimensional semisimple algebra over $\mathbb{C}$ does not have an integral RBA$^{\delta}$-basis. 

\section{The noncommutative $5$-dimensional semisimple algebra}

The results of Section 4 do not tell us if the algebras $\mathbb{C} \oplus M_n(\mathbb{C})$ for $n \ge 2$ under the conjugate-transpose involution have non-integral RBA-structures that admit a positive degree map.  In this section we will consider this question for the $5$-dimensional algebra $A = \mathbb{C} \oplus M_2(\mathbb{C})$.  

\begin{lemma} \label{involu} Let $A = \mathbb{C} \oplus M_2(\mathbb{C})$, and let $\delta$ be the algebra projection map onto its one-dimensional component.  Suppose $\mathbf{B}$ is an RBA-basis of $A$ for which $\delta$ takes positive values on $\mathbf{B}$; i.e.~$\delta$ is a positive degree map.  Then the algebra  ${\mathbb R}\mathbf{B}$ is isomorphic to ${\mathbb R}\oplus M_2({\mathbb R})$ and, up to a change of basis, ${}^*$ acts on $M_2({\mathbb R})$ as matrix transposition.  In particular, $\mathbf{B}$ has exactly three ${}^*$-fixed elements.
\end{lemma}  

\begin{proof} 
By rescaling we can assume $\mathbf{B} = \{b_0=1,b_1,b_2,b_3,b_4\}$ is a standardized RBA$^{\delta}$-basis of $A$.  Set $\delta(b_i)=\delta_i$, and let $n$ be the order of $\mathbf{B}$, so $n = 1 + \delta_1 + \delta_2 + \delta_3 + \delta_4$.   
Since $A$ is non-commutative, the basis $\mathbf{B}$ contains at least one pair $b_i,b_i^*$ of non-symmetric elements. Therefore the number of ${}^*$-fixed elements of ${\mathbf B}$ is either $1$ or $3$. In the first case the dimension of ${}^*$-fixed subspace of $\mathbb{R}\mathbf{B}$ is $3$ while in the second one it is equal to $4$.

The algebra ${\mathbb R}\mathbf{B}$ is a non-commutative semisimple algebra over the reals of dimension $5$. Therefore either ${\mathbb R}\mathbf{B}\cong {\mathbb R}\oplus{\mathbb H}$ or ${\mathbb R}\mathbf{B}\cong \mathbb{R} \oplus M_2({\mathbb R})$.  If $\mathbb{R}\mathbf{B}\cong\mathbb{R}\oplus \mathbb{H}$, then by the Skolem-Noether theorem the action of ${}^*$ on $\mathbb{H}$ has the following form: $x^* = h^{-1} \bar{x} h$ for some unit quaternion $h\in\mathbb{H}$ (here $\bar{x}$ is the standard quaternion conjugation). It follows from $(x^*)^*=x$ that $h^2$ is a scalar quaternion. Therefore $h$  is either scalar (i.e.~an element of $Z(\mathbb{H})\cong\mathbb{R}$) or purely imaginary quaternion.  In the first case $x^*=\bar{x}$ for all $x \in \mathbb{H}$, implying that the dimension of ${}^*$-fixed subspace in $\mathbb{R}\mathbf{B}$ is two which is impossible.  In the second case, we have that $h^*=-h$.   Since $h$ is a purely imaginary unit quaternion, there exists non-zero purely imaginary unit quaternion $q \in \mathbb{H}$ such that $\bar{q}h = - h\bar{q}$.  For this $q$ we have $q^*=-\bar{q}=	q$.   Let $\chi$ be the character of the unique irreducible $\mathbb{H}$-module up to isomorphism; that is, $\chi(1_\mathbb{H})=4,\chi(i)=\chi(j)=\chi(k)=0$.   Then $\tau(x)=\delta(x)+m_\chi \chi(x)$, for all $x \in \mathbb{R}\mathbf{B}$, where $m_\chi = \frac{n-1}{4}$.  
We know $\chi(x^* x)\geq 0$ for each $x\in\mathbb{H}$, but $\chi(q^* q) = \chi(-\bar{q} q) < 0$, a contradiction.  This excludes the case of $\mathbb{R}\mathbf{B}\cong\mathbb{R}\oplus\mathbb{H}$, so we must have that $\mathbb{R}\mathbf{B} \simeq \mathbb{R} \oplus M_2(\mathbb{R})$.

Let $\Delta:{\mathbb R}\mathbf{B}\rightarrow M_2({\mathbb R})$ be the two-dimensional irreducible representation of ${\mathbb R}\mathbf{B}$ given by projection to the component $M_2({\mathbb R})$.  Let $\chi$ be the character corresponding to this representation.  By Proposition \ref{ctorba}, $n = \delta(1) + m_{\chi}\chi(1)$, so $m_{\chi}=\frac{n-1}{2}$.  We have that $\Delta(x^*)^\top$ is a $2$-dimensional irreducible representation equivalent to $\Delta$.  Thus there exists an $S \in GL_2({\mathbb R})$ such that $\Delta(x^*)^\top = S^{-1} \Delta(x) S$. Equivalently, $\Delta(x^*) = S^\top \Delta(x)^\top (S^{-1})^\top$. Substituting $x^*$ intead of $x$ we obtain that $\Delta(x) = (S^\top S^{-1}) \Delta(x) (S^\top S^{-1} )^{-1}$ holds for each $x\in {\mathbb R}\mathbf{B}$. Combining this together with $\Delta({\mathbb R}\mathbf{B}) = M_2({\mathbb R})$ we obtain that $S^{-1}S^\top = \alpha I_2$ for some $\alpha \in {\mathbb R}$. It follows from $S^\top =\alpha S$ and  $(S^\top)^\top = S$ that $\alpha =\pm 1$, i.e.~$S$ is either symmetric or antisymmetric.

Assume first that $S$ is antisymmetric, that is $S = \begin{bmatrix} 0 & a \\ -a & 0 \end{bmatrix}$. A direct check shows that in this case the map $X\mapsto S X^\top S^{-1}, X\in M_2({\mathbb R})$ has a one-dimensional space of fixed points. So, in this case the dimension of the ${}^*$-fixed subspace of $\mathbb{R}\mathbf{B}$ is two, a contradiction.

Assume now that $S$ is symmetric. Then $S=P^\top D P$ for some $D \in \{I_2, Diag(1,-1),-I_2\}$ and 
$P\in GL_2({\mathbb R})$.  Replacing $\Delta(x)$ by the equivalent representation $\Sigma(x):=(P^{-1})^\top \Delta(x) P^\top$ we obtain  $\Sigma(x^*)= D \Sigma(x)^\top D^{-1}$.  If $D=\pm I_2$, then we are done.  It remains to deny the case of $D=Diag(1,-1)$.  

We know that $tr(\Delta(x)\Delta(x^*)) = \chi(xx^*) \geq 0$ for all $x\in {\mathbb R}\mathbf{B}$.  Since $\Sigma:{\mathbb R}\mathbf{B}\rightarrow M_2({\mathbb R})$ is an epimorphism, we conclude that $tr(X D X^\top D^{-1})\geq 0$ holds for all $X\in M_2(\mathbb{R})$. Now choosing $X=\begin{bmatrix} 0 & 1 \\ 1 & 0 \end{bmatrix}$ we get a contradiction.
\end{proof} 

The above lemma tells us that, since $\mathbb{R}\mathbf{B} \simeq \mathbb{R} \oplus M_2(\mathbb{R})$, we can replace $A$ by an isomorphic image whose standardized RBA$^{\delta}$-basis is of the form 
$$\mathbf{B} = \{b_0=(1,I_2), b_1=(\delta_1,B_1)=b_1^*, b_2=(\delta_2,B_2)=b_2^*, b_3 = (\delta_3,B_3), b_4 = b_3^* = (\delta_3,B_3^{\top}) \}, $$ 
and all entries of the matrices $B_1$, $B_2$, and $B_3$ are real. Label the entries of the matrices $B_1$, $B_2$, and $B_3$ so that 
$$ B_1 = \begin{bmatrix} a & b \\ b & d \end{bmatrix}, B_2 = \begin{bmatrix} v & w \\ w & x \end{bmatrix}, B_3 = \begin{bmatrix} r & s \\ t & u \end{bmatrix}, \mbox{ and } B_4=B_3^{\top}. $$   
The centrally primitive idempotents of $A$ are $e_{\delta}=(1,\mathbf{0})$ and $e_{\chi} = (0,I)$.  By Proposition \ref{ctorba} we have   
$$ (1,\mathbf{0}) = \frac{1}{n}\sum_i (\delta_i,B_i) $$ 
and 
$$ (0,I) = \frac{m_{\chi}}{n} \sum_i tr(B_i^{\top})(1,\frac{1}{\delta_i}B_i). $$
These give us the conditions $\sum_i B_i = 0$ and $\sum_i \frac{tr(B_i^{\top})}{\delta_i}B_i = \frac{n}{m_{\chi}}I$.  Since $(1,0)+(0,I)=b_0$, the coefficient of $b_i$ in $(0,I)$ for $i>0$ must be the negative of its coefficient in $(1,0)$.  Therefore, $\frac{-1}{n} = \frac{m_{\chi}tr(B_i^{\top})}{n\delta_i}$ for $i>0$, and hence $\frac{tr(B_i^{\top})}{\delta_i}=\frac{-2}{n-1}$ for $i>0$.  So our character-theoretic identities are:  
$$\begin{array}{ll}
 & 1 + a + v + 2r = 0, \\
 & 1 + d + x + 2u = 0, \\
 & b + w + s + t = 0, \\
 & 2 + \frac{(a+d)}{\delta_1}a + \frac{(v+x)}{\delta_2}v + \frac{2(r+u)}{\delta_3}r = \frac{n}{m_{\chi}}, \\
 & 2 + \frac{(a+d)}{\delta_1}d + \frac{(v+x)}{\delta_2}x + \frac{2(r+u)}{\delta_3}u = \frac{n}{m_{\chi}}, \mbox{ and }\\
 & \frac{(a+d)}{\delta_1} = \frac{(v+x)}{\delta_2} = \frac{(r+u)}{\delta_3} = \frac{-2}{n-1}.
\end{array}$$

The conditions for linear independence of $\mathbf{B}$ and these equations imply that $a \ne d$ or $v \ne x$, at least one of $b$ or $w$ is nonzero, and $s \ne t$.  By Lemma \ref{involu}, we can apply a change of basis to diagonalize the symmetric matrix $B_1$ and assume $b=0$.  

We are able to produce RBA$^{\delta}$-bases with real matrix entries that satisfy all of these conditions.  The main result of this section describes all of these matrix entries in terms of the degrees of basis elements and some sign choices.  Since $\mathbb{R}\mathbf{B} \simeq \mathbb{R} \oplus M_2(\mathbb{R})$, this theorem characterizes all standardized $RBA^{\delta}$-bases of $\mathbb{C} \oplus M_2(\mathbb{C})$ up to equivalence.  

\begin{thm}\label{nc5d}
Suppose
$$ \big\{ (1,I), (\delta_1,\begin{bmatrix} a & 0 \\ 0 & d \end{bmatrix}), (\delta_2, \begin{bmatrix} v & w \\ w & x \end{bmatrix}), (\delta_3, \begin{bmatrix} r & s \\ t & u \end{bmatrix}, (\delta_3, \begin{bmatrix} r & t \\ s & u \end{bmatrix}) \big\} $$
is a standardized RBA$^{\delta}$-basis of $\mathbb{C} \oplus M_2(\mathbb{C})$ with respect to the conjugate-transpose involution, all of whose matrix entries are real.  Let $\varepsilon_1, \varepsilon_2, \varepsilon_3 = \pm 1$ be three sign choices.  Then the matrix entries satisfy the identities

$$\begin{array}{rclcrcl} 
a &=& \displaystyle{-\frac{\delta_1}{n-1} + \varepsilon_1\frac{\sqrt{n\delta_1(n-1-\delta_1)}}{n-1}}, & \qquad & d &=& \displaystyle{-\frac{\delta_1}{n-1} - \varepsilon_1\frac{\sqrt{n\delta_1(n-1-\delta_1)}}{n-1}}, \\
& & & & & & \\
v &=& \displaystyle{-\frac{\delta_2}{n-1} - \varepsilon_1\frac{n\delta_1\delta_2}{ (n-1)\sqrt{n\delta_1(n-1-\delta_1) }}}, &  & 
x &=& \displaystyle{-\frac{\delta_2}{n-1} + \varepsilon_1\frac{n\delta_1\delta_2}{ (n-1)\sqrt{n\delta_1(n-1-\delta_1) }}}, \\
w &=& \displaystyle{\varepsilon_2 \sqrt{\frac{2\delta_2\delta_3}{(n-1)(n-1-\delta_1)}}}, & & & & \\
\end{array}$$

$$\begin{array}{rclcrcl} 
r &=& \displaystyle{-\frac{\delta_3}{n-1} - \varepsilon_1\frac{n\delta_1\delta_3}{ (n-1)\sqrt{n\delta_1(n-1-\delta_1) }}}, & & 
u &=& \displaystyle{-\frac{\delta_3}{n-1} + \varepsilon_1\frac{n\delta_1\delta_3}{ (n-1)\sqrt{n\delta_1(n-1-\delta_1) }}}, \\
& & & & & & \\
s &=& \displaystyle{-\frac{w}{2} + \varepsilon_3 \sqrt{\frac{\delta_3 n}{2(n-1)}}}, & \mbox{ and } & 
t &=& \displaystyle{-\frac{w}{2} - \varepsilon_3 \sqrt{\frac{\delta_3 n}{2(n-1)}}}.
\end{array}$$

Conversely, given positive real numbers $n$, $\delta_1$, $\delta_2$, and $\delta_3$ satisfying $n=1+\delta_1+\delta_2+2\delta_3$ and three choices of sign for $\varepsilon_1$, $\varepsilon_2$, and $\varepsilon_3$, the above identities produce an RBA$^{\delta}$-basis of $\mathbb{C} \oplus M_2(\mathbb{C})$ having real matrix entries. 
\end{thm} 

Before beginning the proof of this theorem, we establish some preliminaries. Let $\tau(\sum_{i} x_i b_i) = n x_0, x = \sum_{i} x_i b_i\in A$ be the standard feasible trace of $A$.  Notice that 
$\tau(x) = \delta(x) +\frac{n-1}{2}\chi(x)$. We denote by $B(x)$ the $2$-dimensional matrix corresponding to the character $\chi$ and by $r(x),s(x)$ the eigenvalues of $B(x)$, for all $x \in A$. Clearly $\chi(x) = r(x)+s(x)$.

\begin{lemma}\label{040315d}
For each $x \in A$ we have $x^2\in span_{\mathbb{C}}( b_0,x,\mathbf{B}^+ )$, where $\mathbf{B}^+=b_0+...+b_4$.
\end{lemma}

\begin{proof} 
The ideal $(b_0 - n^{-1}\mathbf{B}^+)A$ is isomorphic to $M_2(\mathbb C)$. Since
any matrix $B\in M_2(\mathbb C)$ satisfies the identity $B^2 = tr(B) B - \det(B) I_2$, we conclude that
$$
((b_0 - n^{-1}\mathbf{B}^+)x)^2 = \chi(x) (b_0 - n^{-1}\mathbf{B}^+)x + \frac{1}{2}((\chi(x)^2 - \chi(x^2)) (b_0 - n^{-1}\mathbf{B}^+).
$$
Since $\mathbf{B}^+z = \delta(z)\mathbf{B}^+$ for $z \in A$, after opening the brackets and collecting coefficients we obtain the result. (Here we used the identity $\det(B) = \frac{1}{2}(tr(B)^2 - tr(B^2))$.)
\end{proof}

As a corollary we obtain that for any $x \in A$ such that $b_0$, $x$, and $\mathbf{B}^+$ are linearly independent, there exist uniquely determined numbers $\kappa(x),\lambda(x),\mu(x)$ such that 
\begin{equation}\label{040315e}
x^2 = \kappa(x) b_0 +\lambda(x) x + \mu(x) (\mathbf{B}^+ - b_0 - x)
\end{equation}
Let us take $x \in A$ with $\tau(x)=0$ (that is $b_0$ does not appear in $x$).
Then comparing the coefficient of $b_0$ (=applying $n^{-1}\tau$) in both sides gives us
$ \kappa(x) =n^{-1}\tau(x^2) = \langle x,x^*\rangle$.
Applying the degree homomorphism we get 
$\delta(x)^2 =\kappa(x) + \lambda(x)\delta(x) + \mu(x)(n-1-\delta(x))$.

It follows from \eqref{040315e} that $B(x)^2=(\lambda(x)-\mu(x))B(x)+(\kappa(x)-\mu(x))I_2$. Hence
\begin{equation}\label{040315b}
\begin{array}{c}
r(x) + s(x) = \lambda(x) -\mu(x);\\
r(x) s(x)  = \mu(x) - \kappa(x);
\end{array}
\end{equation}
Also $0=\tau(x) = \delta(x) + \frac{n-1}{2}(r(x)+s(x))$ and 
$\kappa(x) n = \tau(x^2) = \delta(x)^2 + \frac{n-1}{2}(r(x)^2+s(x)^2).$
This implies 
\begin{equation}\label{040315f}
\begin{array}{rcl}
r(x)+s(x)& = & -\frac{2\delta(x)}{n-1}\\
r(x)^2+s(x)^2 & = & 2\frac{\kappa(x)n-\delta(x)^2}{n-1}
\end{array}.
\end{equation}
From here we conclude that
$$
r(x)s(x) = \frac{(n+1)\delta(x)^2 - \kappa(x)n(n-1)}{(n-1)^2} \implies 
\mu(x) = \frac{(n+1)\delta(x)^2 - \kappa(x)(n-1)}{(n-1)^2}.
$$
Finally we obtain
\begin{equation}\label{040315c} 
\begin{array}{c}
\lambda(x) = \displaystyle{\frac{(n+1)\delta(x)^2 - 2(n-1)\delta(x) -\kappa(x)(n-1)}{(n-1)^2}}, \mbox{ and } \\
\{r(x),s(x)\} = -\displaystyle{\frac{\delta(x)}{n-1}\pm\frac{\sqrt{\kappa(x) n (n-1)-\delta(x)^2 n}}{n-1}}.
\end{array}
\end{equation}
If $x=\sum\limits_{i=1}^4k_ib_i$, then $\kappa(x)=\langle x,x^* \rangle$ implies that 
$$ \kappa(x)=k_1^2\delta_1 + k_2^2\delta_2 + 2k_3k_4\delta_3. $$
If $x=b_i,i=1,2$, then  $\kappa(b_i)=\delta_i$, 
\begin{equation}\label{120315a}
\lambda(b_i)=\frac{(n+1)\delta_i^2 - 3(n-1)\delta_i}{(n-1)^2}, \mbox{ and }  
\mu(b_i) = \frac{(n+1)\delta_i^2 - \delta_i(n-1)}{(n-1)^2}. 
\end{equation}
If $x=b_i$, $i=3,4$, then $\kappa(b_i)=0$, and hence
\begin{equation}\label{120315b}
\lambda(b_i)= \frac{(n+1)\delta_i^2 - 2(n-1)\delta_i}{(n-1)^2}, \mbox{ and } 
\mu(b_i) =\frac{(n+1)\delta_i^2 }{(n-1)^2}. 
\end{equation}

Now it follows from the above that if $x_0 = 0$ then
$$
x^2 = (r(x)+s(x)) x  - r(x)s(x)b_0 + \mu(x)\mathbf{B}^+ \implies 
$$
\begin{equation}\label{060315a}
B(x)^2 = -\frac{2\delta(x)}{n-1} B(x) - \frac{(n+1)\delta(x)^2 - \kappa(x)n(n-1)}{(n-1)^2} I_2.
\end{equation}
Taking into account that $\kappa(x) = \langle x, x^*\rangle$ we conclude that 
\begin{equation}\label{060315b}
\begin{array}{l}
B(x_1)B(x_2)+B(x_2)B(x_1) = B(x_1+x_2)^2 - B(x_1)^2 - B(x_2)^2 \\
\qquad = -\displaystyle{\frac{2\delta(x_1)}{n-1} B(x_2) - \frac{2\delta(x_2)}{n-1} B(x_1) - 
\frac{2(n+1)\delta(x_1)\delta(x_2) -  2n(n-1)\langle x_1,x_2^*\rangle }{(n-1)^2} I_2}. 
\end{array}
\end{equation}

\medskip
\noindent {\bf Proof of Theorem \ref{nc5d}:}  First, we substitute our RBA-basis elements into the above to establish our identities for the matrix entries. 

\smallskip
{\bf Step 1.} Substituting $x=b_1$ into \eqref{040315c} (notice that $\kappa(b_1)=\delta_1$) we obtain that 
$$
\{r(b_1),s(b_1)\} = -\frac{\delta_1}{n-1} \pm \frac{\sqrt{\delta_1n(n-1)-\delta_1^2n}}{n-1}  \implies 
\{a,d\} = -\frac{\delta_1}{n-1} \pm \frac{\sqrt{\delta_1n(n-1)-\delta_1^2n}}{n-1}\implies
$$
\begin{equation}\label{020515a}
a = -\frac{\delta_1}{n-1} + \varepsilon_1\frac{\sqrt{\Delta_1}}{n-1} \quad \mbox{ and } \quad  
d = -\frac{\delta_1}{n-1} - \varepsilon_1\frac{\sqrt{\Delta_1}}{n-1}, 
\end{equation}
where 
$\Delta_1:=\delta_1n(n-1)-\delta_1^2n = n(n-1-\delta_1)\delta_1$ and $\varepsilon_1=\pm 1$.

\smallskip
{\bf Step 2.} Substituting $x_1 = b_1, x_2 = b_2$ into \eqref{060315b} we obtain
$$
\begin{bmatrix} 2av & (a+d)w\\ (a+d)w & 2dx\end{bmatrix}
=
-\frac{2\delta_1}{n-1} 
\begin{bmatrix} v & w\\ w & x\end{bmatrix}
 - \frac{2\delta_2}{n-1} 
\begin{bmatrix} a & 0\\ 0 & d\end{bmatrix} - 
\frac{2(n+1)\delta_1\delta_2}{(n-1)^2} I_2.
$$
From the above equation we can derive equations for $v$ and $x$: 
$$
\left\{
\begin{array}{lcr}
av & = & -\frac{\delta_1}{n-1} v - \frac{\delta_2}{n-1}a - \frac{(n+1)\delta_1\delta_2}{(n-1)^2}\\
dx & = & -\frac{\delta_1}{n-1} x - \frac{\delta_2}{n-1}d - \frac{(n+1)\delta_1\delta_2}{(n-1)^2}
\end{array}
\right.
$$
Thus, substituting the values of $a$ and $d$ given in \eqref{020515a} into the above equations it is straightforward to check that 
$$
\left\{
\begin{array}{lcr}
 v & = &  -\frac{\delta_2}{n-1} - \varepsilon_1\frac{n\delta_2\delta_1}{ (n-1)\sqrt{\Delta_1}}  \\
x & = &   -\frac{\delta_2}{n-1} + \varepsilon_1\frac{n\delta_2\delta_1}{ (n-1)\sqrt{\Delta_1}}
\end{array}
\right.
$$

\smallskip
{\bf Step 3.} Substituting $b_2$ for $x$ into \eqref{060315a} we obtain 
$$ w^2 =-v^2 -\frac{2 \delta_2 v}{n-1} -\frac{(n+1)\delta_2^2-n(n-1)\delta_2}{(n-1)^2}. $$ 
Substiting the value of $v$ obtained in Step 2 yields 
$$ w^2 = \frac{2n\delta_2\delta_3}{(n-1)(n-1-\delta_1)}. $$
So 
\begin{equation}\label{110315c}
w = \varepsilon_2 \sqrt{\frac{2n\delta_2\delta_3}{(n-1)(n-1-\delta_1)}}
\end{equation}
where $\varepsilon_2 =\pm 1$.

\smallskip
{\bf Step 4.}  Substituting the values for $a$, $v$, $d$, and $x$ obtained above into the equations $1+a+v+2r=0$ and $1+d+x+2u=0$ gives the indicated values of $r$ and $u$. 

\smallskip
{\bf Step 5.} Substituting $x=b_3-b_4$ into \eqref{060315a} 
and taking into account that $\delta(b_3-b_4)=0, \kappa(b_3-b_4)=\langle b_3 - b_4, b_4 - b_3\rangle =-2\delta_3$, 
we obtain
$$
\begin{bmatrix} 0 & s-t \\ t-s & 0 \end{bmatrix}^2 = -\frac{2\delta_3 n}{n-1} I_2\implies s-t =  \varepsilon_3\sqrt{\frac{2\delta_3 n}{n-1}},
$$
where $\varepsilon_3=\pm 1$.

\smallskip
{\bf Step 6.} Using $w+s+t =0$ we find that 
$$
s = -\frac{w}{2} + \varepsilon_3 \sqrt{\frac{\delta_3 n}{2(n-1)}},
t = -\frac{w}{2} - \varepsilon_3 \sqrt{\frac{\delta_3 n}{2(n-1)}}.
$$

This completes the proof in one direction.  The other direction can be proved in a straightforward (although tedious) manner by simply calculating structure constants for the basis.  Our formulas for these structure constants are (with $\varepsilon:=\varepsilon_1\varepsilon_2\varepsilon_3$, and not including those involving $b_0$): 

$$\begin{array}{ll} 
\lambda_{111} = \frac{(n+1)\delta_1^2-3(n-1)\delta_1}{(n-1)^2}, \quad & \lambda_{112}=\lambda_{113}=\lambda_{114}= \frac{(n+1)\delta_1^2-(n-1)\delta_1}{(n-1)^2}, \\
\lambda_{121}=\lambda_{211}=\frac{(n+1)\delta_1\delta_2-(n-1)\delta_2}{(n-1)^2}, \quad & \lambda_{122}=\lambda_{212}=\frac{(n+1)\delta_1\delta_2-(n-1)\delta_1}{(n-1)^2}, \\
\lambda_{123}=\lambda_{214}=\frac{(n+1)\delta_1\delta_2+\varepsilon(n-1)\sqrt{n\delta_1\delta_2}}{(n-1)^2}, \quad & \lambda_{124}=\lambda_{213}=\frac{(n+1)\delta_1\delta_2-\varepsilon(n-1)\sqrt{n\delta_1\delta_2}}{(n-1)^2}, \\

\lambda_{131}=\lambda_{141}=\lambda_{311}=\lambda_{411}=\frac{(n+1)\delta_1\delta_3-(n-1)\delta_3}{(n-1)^2}, \quad & \lambda_{132}=\lambda_{412}=\frac{(n+1)\delta_1\delta_2\delta_3+\varepsilon(n-1)\delta_3\sqrt{n\delta_1\delta_2}}{\delta_2(n-1)^2}, \\
\lambda_{133}=\lambda_{414}=\frac{(n+1)\delta_1\delta_3-(n-1)\delta_1-\varepsilon(n-1)\sqrt{n\delta_1\delta_2}}{(n-1)^2}, \quad & \lambda_{134}=\lambda_{143}=\lambda_{314}=\lambda_{413}=\frac{(n+1)\delta_1\delta_3}{(n-1)^2} \\
\lambda_{142}=\lambda_{312}=\frac{(n+1)\delta_1\delta_2\delta_3-\varepsilon(n-1)\delta_3\sqrt{n\delta_1\delta_2}}{\delta_2(n-1)^2}, \quad & \lambda_{144}=\lambda_{313}=\frac{(n+1)\delta_1\delta_3-(n-1)\delta_1+\varepsilon(n-1)\sqrt{n\delta_1\delta_2}}{(n-1)^2}, \\

\lambda_{221}=\lambda_{223}=\lambda_{224}=\frac{(n+1)\delta_2^2-(n-1)\delta_2}{(n-1)^2}, \quad & \lambda_{222}=\frac{(n+1)\delta_2^2-3(n-1)\delta_2}{(n-1)^2}, \\

\lambda_{231}=\lambda_{421}=\frac{(n+1)\delta_1\delta_2\delta_3-\varepsilon(n-1)\delta_3\sqrt{n\delta_1\delta_2}}{\delta_1(n-1)^2}, \quad & \lambda_{232}=\lambda_{422}=\lambda_{242}=\lambda_{322}=\frac{(n+1)\delta_2\delta_3-(n-1)\delta_3}{(n-1)^2}, \\
\lambda_{233}=\lambda_{424}=\frac{(n+1)\delta_2\delta_3-(n-1)\delta_2+\varepsilon(n-1)\sqrt{n\delta_1\delta_2}}{(n-1)^2}, \quad & \lambda_{234}=\lambda_{423}=\lambda_{243}=\lambda_{324}=\frac{(n+1)\delta_2\delta_3}{(n-1)^2} \\

\lambda_{241}=\lambda_{321}=\frac{(n+1)\delta_1\delta_2\delta_3+\varepsilon(n-1)\delta_3\sqrt{n\delta_1\delta_2}}{\delta_1(n-1)^2}, \quad & \lambda_{244}=\lambda_{323}=\frac{(n+1)\delta_2\delta_3-(n-1)\delta_2-\varepsilon(n-1)\sqrt{n\delta_1\delta_2}}{(n-1)^2}, \\
\lambda_{331}=\lambda_{332}=\lambda_{334}= \frac{(n+1)\delta_3^2}{(n-1)^2}, \quad & \lambda_{441}=\lambda_{442}=\lambda_{443}=\frac{(n+1)\delta_3^2}{(n-1)^2}, \\
\lambda_{343}=\lambda_{344}=\lambda_{433}=\lambda_{434}=\frac{(n+1)\delta_3^2-2(n-1)\delta_3}{(n-1)^2} \quad & 
\lambda_{333}=\lambda_{444}=\frac{(n+1)\delta_3^2-2(n-1)\delta_3}{(n-1)^2}, \\
\lambda_{341}=\lambda_{432}=\frac{(n+1)\delta_1\delta_3^2-(n-1)\delta_1\delta_3-\varepsilon(n-1)\delta_3\sqrt{n\delta_1\delta_2}}{\delta_1(n-1)^2}, \quad & \lambda_{342}=\lambda_{431}=\frac{(n+1)\delta_1\delta_3^2-(n-1)\delta_1\delta_3+\varepsilon(n-1)\delta_3\sqrt{n\delta_1\delta_2}}{\delta_1(n-1)^2}.
\end{array}$$

\hfil \hfil \hfil \hfil \hfil \hfil \hfil $\square$

\medskip
Any combination of the three sign choices in Theorem \ref{nc5d} is interchangeable by a change of basis.  This is because conjugation by $\begin{bmatrix} 0 & 1 \\ 1 & 0 \end{bmatrix}$ interchanges the sign choices for $\varepsilon_1$ and $\varepsilon_3$ and fixes the one for $\varepsilon_2$, and conjugating by $\begin{bmatrix} 1 & 0 \\ 0 & -1 \end{bmatrix}$ switches the choices for $\varepsilon_2$ and $\varepsilon_3$ and fixes that of $\varepsilon_1$.  So any combination of the three sign choices can be achieved through a sequence of these operations.  

\medskip
We have used Theorem \ref{nc5d} to look for RBA$^{\delta}$-bases of $\mathbb{C} \oplus M_2(\mathbb{C})$ that have rational and/or nonnegative structure constants.  Both situations occur, and can occur simultaneously.  

\begin{example} \label{e1} {\rm Choosing $n=25$, $\delta_1=\delta_2=\delta_3=6$ with any choice of the three signs produces a {\it rational} TA-basis for $\mathbb{C} \oplus M_2(\mathbb{C})$.  The basis elements (with positive sign choices) are: $b_0 =(1,I)$,
$$ b_1 = (6, \begin{bmatrix} \frac{-1+5\sqrt{3}}{4} & 0 \\ 0 & \frac{-1-5\sqrt{3}}{4} \end{bmatrix}), \quad b_2 = (6, \begin{bmatrix} \frac{-3-5\sqrt{3}}{12} & \frac{5}{\sqrt{6}} \\ \frac{5}{\sqrt{6}} & \frac{-3+5\sqrt{3}}{12} \end{bmatrix}), \quad b_3 = (6, \begin{bmatrix} \frac{-3-5\sqrt{3}}{12} & \frac{-5\sqrt{6}+3\sqrt{2}}{12} \\ \frac{-5\sqrt{6}-3\sqrt{3}}{12} & \frac{-3+5\sqrt{3}}{12} \end{bmatrix}), $$
and $b_3^*$.  All structure constants for this basis lie in $\mathbb{Z}[\frac12]$, and the largest denominator that occurs among them is an $8$.
}\end{example} 

\begin{example} {\rm Here are the elements of another RBA$^{\delta}$-basis of $\mathbb{C} \oplus M_2(\mathbb{C})$ whose entries and structure constants are all rational:  $b_0 =(1,I)$,
$$ b_0 = (1,I_2), \quad b_1 = \big(\frac32, \begin{bmatrix} -\frac{3}{2} & 0 \\ 0 & \frac12 \end{bmatrix}\big), \quad b_2 = \big(\frac16, \begin{bmatrix} \frac29 & \frac49 \\ \frac49 & -\frac{1}{6} \end{bmatrix}\big), \quad b_3 = \big(\frac23, \begin{bmatrix} \frac29 & \frac49 \\ -\frac{8}{9} & -\frac{2}{3} \end{bmatrix}\big), \mbox{ and }b_3^*.$$  
}\end{example} 

\section{RBA$^{\delta}$ bases for $\mathbb{C} \oplus M_n(\mathbb{C})$, $n \ge 2$} 

\medskip
In light of Theorem \ref{nc5d} one is almost certain that  $\mathbb{C} \oplus M_n(\mathbb{C})$ will have an RBA$^{\delta}$-basis for $n > 2$.  In this section we give a general construction that applies for all $n \ge 2$. 

We begin with a general restriction on involutions admitting positive degree maps that applies to any semisimple algebra.  

\begin{thm}\label{conjtransthm}
Let $A$ be a finite-dimensional semisimple algebra whose involution $*$ extends complex conjugation on scalars.  Suppose $\chi \in Irr(A)$, and let $Ae_{\chi}$ be the simple component of $A$ corresponding to $\chi$.  Identify $Ae_{\chi}$ with a full matrix algebra $M_m(\mathbb{C})$ where $m = \chi(1)$.  If $A$ has an RBA-basis that admits a positive degree map, then up to a change of basis, the restriction of $*$ to $Ae_{\chi}$ will be equal to the conjugate transpose map on $M_m(\mathbb{C})$.   
\end{thm} 

\begin{proof}
Suppose $\mathbf{B} = \{1=b_0,b_1,\dots,b_d\}$ is an RBA-basis of $A$ that admits a positive degree map $\delta$.  Let $\tau$ be the standard feasible trace of $A$. 
By Proposition \ref{ctorba} and the subsequent remarks, $\tau = \sum_{\psi \in Irr(A)} m_{\psi} \psi$ with $m_{\delta}=1$ all $m_{\psi}>0$, and $\psi(xx^*)\ge 0$, for all $\psi \in Irr(A)$.  Consider the restriction of the involution $*$ to $Ae_{\chi}$, which we identify with $M_m(\mathbb{C})$.  Since the projection of $A$ into this simple component is surjective, the above implies that $tr(XX^*) \ge 0$ for all $X \in M_m(\mathbb{C})$.   Since $*$ extends complex conjugation on scalars, the map $X \mapsto \overline{(X^*)}^\top$ is an algebra automorphism of $M_m(\mathbb{C})$.  Therefore, $X^* = S^{-1} \overline{X}^\top S$, for some $S \in GL_m(\mathbb{C})$.  Replacing $X$ by $X^*$ we obtain $X = (S^{-1}\overline{S}^\top) X (S^{-1} \overline{S}^\top)^{-1}$, for all $X \in M_m(\mathbb{C})$, and so it follows that $S^{-1}\overline{S}^\top = \alpha I_n$, for some nonzero $\alpha \in \mathbb{C}$.  

From the equation $\overline{S}^{\top} = \alpha S$ we have that $S$ commutes with its conjugate transpose, and so $S^* = S^{-1} \overline{S}^{\top} S = \overline{S}^{\top}$.   Furthermore, $S = \overline{(\overline{S}^{\top})}^{\top} = \bar{\alpha} \alpha S$, so $\alpha = e^{i\theta}$ for some $\theta \in \mathbb{R}$.   If we set $H = e^{\frac{i\theta}{2}}S$, then $\overline{H}^{\top} = H = H^*$, and $X^* = H^{-1} \overline{X}^{\top} H$ for all $X \in M_m(\mathbb{C})$.  

We need to find an invertible matrix $P\in GL_m(\mathbb{C})$ such that $P X^* P^{-1} = \overline{(P X P^{-1})}^{\top}$ for all $X \in M_n(\mathbb{C})$.  It is enough to find an invertible matrix $P$ for which $\overline{P}^{\top} P = H$. 

Since $H$ is a non-degenerate Hermitian matrix, $H = U^{-1}DU$ for some unitary matrix $U$ ($U^{-1}=\overline{U}^{\top}$) and diagonal matrix $D$ with nonzero real diagonal entries $\lambda_1,...,\lambda_n$.  For each $1 \le s,t \le n$, let $Y = U^{-1}E_{st}U$, where $E_{st}$ is the matrix unit with $(s,t)$-entry $1$ and $0$ entries elsewhere.  Then 
$$ Y^* Y = H^{-1} \overline{Y}^{\top} H Y = U^{-1} D^{-1} U \overline{Y}^{\top} U^{-1} D U Y = U^{-1}D^{-1}E_{ts}DE_{st}U. $$ \
Hence, 
$$ 0 \le \chi(Y^* Y) = \mathsf{tr}(D^{-1}E_{ts}DE_{st}) = \lambda_t^{-1} \lambda_s. $$
Therefore, all of the eigenvalues of $H$ have the same sign.  Replacing $H$, if necessary,  by $-H$ we may assume that
$\lambda_i > 0$ for all $i=1,...,n$.  Let $\sqrt{D}$ be the diagonal matrix with diagonal entries $\sqrt{\lambda_i}$, $i = 1,\dots,n$.  Let $P = \sqrt{D} U$.  Then $H = \overline{P}^{\top} P$, as required.
\end{proof} 

\medskip
We will finish with a construction of an RBA$^{\delta}$-basis of $\mathbb{C} \oplus M_m(\mathbb{C})$ with respect to the involution that restricts to the conjugate transpose in the second component.  As in the last section, it suffices to find an RBA$^{\delta}$-basis of $A:=\mathbb{R} \oplus M_m(\mathbb{R})$, which is what we will do.   In what follows we write $(X,Y)$ for $tr(XY^\top)$. Notice that $(\star,\star)$ is a standard Euclidean form on the vector space $M_m(\mathbb{R})$. Notice that $(XY,Z)=(Y,X^\top Z)$ for all $X,Y,Z\in M_m(\mathbb{R})$.

\begin{proposition}\label{p1} Let $\delta_1,...,\delta_{m^2}$ be positive real numbers. Set $n:=1+\sum_{i=1}^{m^2} \delta_i$. Assume that there exist $m^2$ matrices $B_1,...,B_{m^2} \in M_m(\mathbb{R})$ such that 
\begin{enumerate}
\item[(a)] 
$$
(B_i,B_j)= \left\{\begin{array}{rl}
\frac{\delta_i(n-\delta_i) m}{n-1}, & i=j \\
-\frac{\delta_i\delta_j m}{n-1}, & i \ne j
\end{array}\right.
$$
\item[(b)]there exists an involutive permutation $i\mapsto i',i = 1,..., m^2$ with exactly $m$ fixed points such that $B_i^\top = B_{i'}$ and $\delta_i = \delta_{i'}$ hold for all $i$;
\item[(c)] $\sum_{i=1}^{m^2} B_i = -I$. 
\end{enumerate}
Then the vectors $b_0 := (1,I), b_i:=(\delta_i, B_i)$ form an $RBA^\delta$-basis of $\tilde{A}=\mathbb{R} \oplus M_m(\mathbb{R})$. 
\end{proposition}

\begin{proof}
Define a bilinear form $\sg{\star,\star}$ on $\tilde{A}$ as follows
$$
\sg{(x,X),(y,Y)}:= n^{-1}(xy +\frac{n-1}{m} (X,Y)).
$$
Also define an anti-automorphism ${}^*$ of $\tilde{A}$ by $(x,X)^* = (x,X^\top)$.
A direct check shows that 
\begin{enumerate}
\item $b_i^* = b_{i'}$;
\item $\sg{b_i,b_j} = \delta_{ij} \delta_i$ (here $\delta_{ij}$ is the Kronecker's delta);
\item $\sg{b_i b_j,b_k} = \sg{b_j,b_{i'} b_k}$
\end{enumerate}
Since $b_0,b_1,...,b_{m^2}$ is a basis of the real algebra $\tilde{A}$, we obtain that
$b_i b_j=\sum_{k}\lambda_{ijk} b_k$. where $\lambda_{ijk}$ are real numbers.
Since $b_i$'s form an orthogonal basis of $\tilde{A}$, we get $\lambda_{ijk} \delta_k = \sg{b_i b_j,b_k}$. 

Since $b_0$ is the identity of $A$ and $\delta_0=1$, we can write
$\delta_{ij} \delta_i = \sg{b_i,b_j} = \sg{b_0, b_{i'} b_j} = \lambda_{i'j0}$. Thus 
$\lambda_{ab0}\neq 0 \iff a' = b$ and in the latter case $\lambda_{aa'0} = \delta_a > 0$.

Thus the basis $b_0,...,b_{m^2}$ satisfies all the axioms of RBA$^\delta$ basis.
\end{proof}

\begin{proposition}\label{p2} Assume that there exist matrices $B_i$, $i=1,...,m^2$ which satisfy conditions (a)-(b) of Prop~\ref{p1}. Denote $B:=\sum_{i=1}^{m^2} B_i$. 
Then the matrices 
$$
\tilde{B}_i:=B_i -2\frac{(B+I, B_i)}{(B+I,B+I)} (B+I)
$$
satisfy the conditions (a)-(c) of Prop~\ref{p1}.
\end{proposition}

\begin{proof} Notice that the linear map $L: X\mapsto \tilde{X}:=X - 2\frac{(B+I, X)}{(B+I,B+I)} (B+I)$ is a reflection with respect to the form $(\star,\star)$.   Therefore $(L(X),L(Y))=(X,Y)$,  implying $(\tilde{B}_i,\tilde{B}_j)=(B_i,B_j)$ hereby proving (a).

It follows from part (a) of Prop~\ref{p1} that $(B,B)=m$. Together with $(I,I)=m$ this implies that $\tilde{B} = - I$. Thus
$\sum_{i}\tilde{B}_i =\tilde{B}=-I$. Thus $\tilde{B}_i$ satisfy the properties (a) and (c). The property (b) follows from $L(X^\top)=L(X)^\top$ (notice that $B+I$ is a symmetric matrix).
\end{proof}

Now we'd like to build a basis $B_{ij}$ of $M_m(\mathbb{R})$ which satisfies conditions (a)-(b) of Prop~\ref{p2}. 
To simplify calculations we take all $\delta_i$ to be the same, i.e. $\delta_k= \delta > 0$. In this case $n = 1 + m^2\delta$ and the conditions in part (a) of Proposition \ref{p2} read as follows:

\begin{equation}\label{eq1}
(B_{ij},B_{k\ell})= \left\{\begin{array}{rl}
\frac{n-\delta}{m}, & (i,j)=(k,\ell) \\
-\frac{\delta}{m}, & (i,j)\neq (k,\ell)
\end{array}\right.
\end{equation}

To build a basis which satisfies (a)-(b) we  look for the matrices of the form $B_{ij} = x E_{ij} + y J$ where $E_{ij}$ are elementary matrices, $J$ the $m \times m$ matrix of all $1$'s, and $x,y$ are real parameters which will be found later. Clearly the matrices $B_{ij}$ satisfy part (b) of Proposition \ref{p2}. To satisfy (a) we first compute the inner products $(B_{ij},B_{k\ell})$. We obtain that
$(B_{ij},B_{ij}) = x^2 + 2xy + y^2 m^2$ and  
$(B_{ij},B_{k\ell}) = 2xy + y^2 m^2$ for $(i,j)\neq (k,\ell)$. Now \eqref{eq1} yields the following equations for $x$ and $y$
$$
x^2 + 2xy + y^2 m^2 = \frac{n-\delta}{m}, \qquad 2xy + y^2 m^2 = -\frac{\delta}{m}
$$
Substracting the second equation from the first we obtain $x= \pm\sqrt{\frac{n}{m}}$.  Then from the second one we get $y =\frac{1}{m^2}\left(-x \pm \frac{1}{\sqrt{m}}\right)$.  Thus there exist matrices which satisfy conditions (a)-(b).  Applying Prop~\ref{p2} we conclude that $A$ has an $RBA^\delta$-basis.  This proves the main result of this section. 

\begin{thm} \label{CpMmC}
The noncommutative algebra $\mathbb{C} \oplus M_m(\mathbb{C})$ with the conjugate transpose involution has an RBA$^{\delta}$-basis for all $m \ge 2$. 
\end{thm}

\begin{corollary} 
Every finite-dimensional semisimple algebra that has a one-dimensional simple component can be equipped with an RBA$^{\delta}$-basis. 
\end{corollary} 

\begin{proof} 
Suppose $A \simeq \mathbb{C} \oplus A_0$.  If $A_0$ is simple then it follows from Theorem \ref{CpMmC} that $A$ has an RBA$^{\delta}$-basis.  Otherwise we can write $A \simeq A_1 \oplus M_m(\mathbb{C})$ where $A_1$ has a one-dimensional component.  By induction on the dimension we can assume $A_1$ has an RBA$^{\delta}$-basis $\mathbf{B}_1$.  Let $\mathbf{B}'$ be an RBA$^{\delta}$-basis of $\mathbb{C} \oplus M_m(\mathbb{C})$.  By applying properties we have for the circle product and arguing as in the proof of Corollary \ref{allrba}, we find that $\mathbf{B}_1 \circ \mathbf{B}'$ is an RBA$^{\delta}$-basis of 
$$A_1 \circ (\mathbb{C} \oplus M_m(\mathbb{C})) \simeq A_1 \oplus M_m(\mathbb{C}) \simeq A.$$  
\end{proof}

\begin{example} {\rm We will illustrate the construction for Theorem \ref{CpMmC} in the case where $m=3$ and $\delta=7$.  In this case the construction says to take $x=-\frac{8}{\sqrt{3}}$ and $y = \frac{1}{\sqrt{3}}$, and set $B_{ij}$ to be the $3 \times 3$ matrix $\frac{1}{\sqrt{3}} (-8E_{ij}+J)$ for $1 \le i,j \le 3$.  Applying the reflection mapping of Proposition \ref{p2} to each $B_{ij}$ produces the following list of matrices $\tilde{B}_{ij}$, for $1 \le i,j \le 3$: 
$$
\tilde{B}_{11} = \frac19 \begin{bmatrix} -1-16\sqrt{3} & 8 & 8 \\ 8 & -1+8\sqrt{3} & 8 \\ 8 & 8 & -1+8\sqrt{3} \end{bmatrix}, 
\tilde{B}_{12} = \frac19 \begin{bmatrix} -1 & -4-20\sqrt{3} & -4+4\sqrt{3} \\ -4+4\sqrt{3} & -1 & -4+4\sqrt{3} \\ -4+4\sqrt{3} & -4+4\sqrt{3} & -1 \end{bmatrix}, $$
$$\begin{array}{l}
\tilde{B}_{13} = \displaystyle{\frac19} \begin{bmatrix} -1 & -4+4\sqrt{3} & -4-20\sqrt{3} \\ -4+4\sqrt{3} & -1 & -4+4\sqrt{3} \\ -4+4\sqrt{3} & -4+4\sqrt{3} & -1 \end{bmatrix}, \tilde{B}_{21} = \tilde{B}_{12}^{\top} \qquad \qquad \qquad \qquad \qquad \qquad \qquad \qquad \qquad \\
\end{array}$$
$$ \tilde{B}_{22} = \frac19 \begin{bmatrix} -1+8\sqrt{3} & 8 & 8 \\ 8 & -1-16\sqrt{3} & 8 \\ 8 & 8 & -1+8\sqrt{3} \end{bmatrix}, 
 \tilde{B}_{23} = \frac19 \begin{bmatrix} -1 & -4+4\sqrt{3} & -4+4\sqrt{3} \\ -4+4\sqrt{3} & -1 & -4-20\sqrt{3} \\ -4+4\sqrt{3} & -4+4\sqrt{3} & -1 \end{bmatrix},  $$
 $$
 \tilde{B}_{31} = \tilde{B}_{13}^{\top}, \quad \tilde{B}_{32} =  \tilde{B}_{23}^{\top}, \mbox{ and } 
\tilde{B}_{33} = \frac19 \begin{bmatrix} -1+8\sqrt{3} & 8 & 8 \\ 8 & -1+8\sqrt{3} & 8 \\ 8 & 8 & -1-16\sqrt{3} \end{bmatrix}. $$  

The set consisting of $(1,I_3)$ and the nine $(7,\tilde{B}_{ij})$'s for $1 \le i,j \le 3$ is (indeed!) an RBA$^{\delta}$-basis of $\mathbb{C} \oplus M_3(\mathbb{C})$ whose structure constants lie in the ring $\mathbb{Z}[\frac{1}{3}, \sqrt{3}]$.  The largest denominator that occurs among its structure constants is a $27$. 
}\end{example} 

\bigskip
The authors would like to express their gratitude to the anonymous referee that gave several insightful comments that significantly improved the final presentation of this article.

{\footnotesize 

\medskip
{\sc A.~Herman, Department of Mathematics and Statistics, University of Regina, Regina, SK S4A 0A2, CANADA. Email:} {\tt Allen.Herman@uregina.ca}

\medskip
{\sc M.~Muzychuk, Department of Computer Science and Mathematics, Netanya Academic College, University St. 1, 42365, Netanya, ISRAEL.  Email:} {\tt muzy@netanya.ac.il}

\medskip
{\sc B.~Xu, Department of Mathematics and Statistics, Eastern Kentucky University, 521 Lancaster Ave., Richmond, KY 40475-3133, USA. Email:} {\tt bangteng.xu@eku.edu} 

\end{document}